\documentclass[10pt]{amsart} 
\usepackage{amsfonts}
\usepackage{amsmath}
\usepackage{amssymb}
\usepackage{graphicx}%
\usepackage{color}

\definecolor{m}{cmyk}{0,1,0,0}

\definecolor{b}{rgb}{0,0,1}
\def\blue#1{{\color{b} #1}}

\definecolor{r}{rgb}{1,0,0}
\def\red#1{{\color{r} #1}}

\newcommand{\NN}{\mathbb{N}}

\newcommand{\RR}{\mathbb{R}}
\newcommand{\CC}{\mathbb{C}}

\newcommand{\HH}{\mathbb{H}}

\newcommand{\rr}{\mathbb{R}}

\providecommand{\U}[1]{\protect\rule{.1in}{.1in}}
\newtheorem{theorem}{Theorem}

\newtheorem{corollary}[theorem]{Corollary}

\newtheorem{definition}[theorem]{Definition}

\newtheorem{lemma}[theorem]{Lemma}

\newtheorem{proposition}[theorem]{Proposition}
\newtheorem{remark}[theorem]{Remark}

\numberwithin{equation}{section}
\numberwithin{theorem}{section}

\begin{document}

\title[Nonlocal heat equations in the Heisenberg group]{\bf Nonlocal heat equations  in the Heisenberg group}

\author[R. Vidal]{Ra\'ul E. Vidal }

\address[R. Vidal]{\newline
\noindent FaMAF, Universidad Nacional de Cordoba, (5000), Cordoba, Argentina.
\newline
\noindent
email: {\tt vidal@mate.uncor.edu}}

\keywords{Nonlocal diffusion, Heisenberg group. \\
\indent 2010 {Mathematics Subject Classification: 47G10, 47J35, 45G10.   }}


\begin{abstract}
We study the following nonlocal diffusion equation in the Heisenberg group $\HH_n$,
\[
u_t(z,s,t)=J\ast u(z,s,t)-u(z,s,t),
\]         
where $\ast$ denote convolution product and $J$ satisfies appropriated hypothesis.
For the Cauchy problem we obtain that the asymptotic behavior of the solutions is the same form that the one for the heat equation in the Heisenberg group. To obtain this result we use the spherical transform related to the pair $(U(n),\HH_n)$. Finally we prove that solutions of properly rescaled nonlocal Dirichlet problem converge uniformly to the solution of the corresponding Dirichlet problem for the classical heat equation in the Heisenberg group.
\end{abstract}

\maketitle

\section{Introduction and Preliminaries} 

\label{sect-intro}

During the last years, many authors have studied the asymptotic behavior for several nonlocal diffusion models in the whole $\RR^n$. In some cases, this behavior is related with the asymptotic behavior of the local diffusion model.  

In  \cite{CCR} the authors study the nonlocal diffusion equation in $\RR^n$ given by
\begin{align}\label{0.1}
u_t(x,t)=J\ast u(x,t)-u(x,t),
\end{align}
where $\ast$ denote convolution product. 
For the Cauchy problem, they prove that the long time behavior of the solutions is determined by the behavior of the Fourier transform $\widehat{J}$ of $J$ near the origin. If $\widehat{J}(\xi ) = 1 - A|\xi|^\alpha + o(|\xi|^\alpha),\,\,\,\,  (0 <\alpha\leq 2)$, the asymptotic behavior is the same as the one for solutions of the evolution given by
the $\alpha/2$ fractional power of the Laplacian. Concerning the Dirichlet problem for the nonlocal model they prove that the asymptotic behavior is given by an exponential decay to zero at a rate given by the
first eigenvalue of an associated eigenvalue problem with profile an eigenfunction of the first eigenvalue. Finally, they analyse the Neumann problem and find an exponential convergence to the mean value of the initial condition.

In the work \cite{CER} the authors prove that solutions of properly rescaled nonlocal Dirichlet problems of the equation (\ref{0.1}) approximate uniformly the solution of the corresponding Dirichlet problem for the classical heat equation in $\rr^n$.

These type of problems have been studied for the case of different elliptical operators and $p$-Laplacian operators, see \cite{andreu}, \cite{andreu2}, \cite{BCC}, \cite{Ignat-Rossi}, \cite{KRV}, \cite{PLPS} and \cite{S1}. 

In \cite{R} the author considers the classic heat equation for Carnot groups and settles the  asymptotic behavior of the solution. The Heisenberg group is the main example of the Carnot groups.      

At the present work we study a similar problems to the ones in \cite{CCR} and \cite{CER}, in the Heisenberg group. In order to do this we have to consider the results obtained in \cite{R}, the fact that $\HH_n$ is a homogeneous group and the harmonic analysis related to the action of the unitary group $U(n)$ by automorphism on $\HH_n$. 

Let $\HH_n=\CC^n\times \RR$ the $2n+1$ dimensional Heisenberg group, with law group $(z,s).(\tilde z, \tilde s)=(z+\tilde z,s+ \tilde s+\frac12\text{Im}\langle z,\tilde z\rangle)$, where $\langle z,\tilde z\rangle$ denote the Hermitian inner product of $\CC^n$.  The Haar measure of the group is de Lebesgue measure.  If we write $z=x+iy$, with $x,y$ in $\rr^n$     
we have a global coordinate system $(x,y,s)$ and the vector fields $X_j=\frac{\partial}{\partial x_{j}}-\frac{y_j}{2}\frac{\partial}{\partial s},$ $Y_j=\frac{\partial}{\partial y_{j}}+\frac{x_j}{2}\frac{\partial}{\partial s},$ and $S=\frac{\partial}{\partial s}$ form a basis for the Lie algebra $\mathfrak{h}_{n}$ of $\HH_n$. 

The Heisenberg Laplacian is $ L := \sum\limits_{j=1}^{n} X_{j}^{2}+Y_{j}^{2}$. In coordinates is given by
\begin{align}\label{0.0.0.1}
L = \sum\limits_{j=1}^{n}\left( \frac{\partial^2}{\partial x_{j}^{2}}+\frac{\partial^2}{\partial y_{j}^{2}}\right)+\frac14\frac{\partial^2}{\partial s^{2}}\sum\limits_{j=1}^{n}\left(x_j^2+y_j^2\right)+\frac{\partial}{\partial s}\sum\limits_{j=1}^{n}\left(x_j\frac{\partial}{\partial y_j}-y_j\frac{\partial}{\partial x_j}  \right),
\end{align} 
The Laplacian $L$ is a second order degenerate elliptic operator of H\"ormander type and hence it is hypoelliptic see \cite{J}. 

We recall that a Lie group is called a \textit{homogeneous group} if it is a connected, simply connected, nilpotent Lie group $G$, whose Lie algebra $\mathfrak{g}$ is endowed with a family of dilatation $\{\delta_r\}_{r\in\NN}$. Let exp$:\mathfrak{g}\rightarrow G$ be the exponential map, which in this case is a diffeomorphism.  The maps exp$\,\delta_r\,$exp$^{-1}$ are group automorphisms of $G$ also denoted by $\delta_r$ and called dilations of $G$. A standard example is given by $\delta_{r}(z,s)=(r^{\frac12}z,rs)$, $r>0$ and $(z,s)\in\HH_n$. 

Let $U(n)$ the unitary group, which acts by automorphism on $\HH_n$ by $g\cdot(z,s)=(gz,s),$ $g\in U(n)$ and $(z,s)\in \HH_n$.
We will denote by $\mathcal{S}(\HH_n)^{U(n)}$ the space of functions in the Schwartz space that are invariant by the action of $U(n)$ and we will denote by $L^1(\HH_n)^{U(n)}$ the space of $L^1(\HH_n)$ the functions that are invariant by the action of $U(n)$. Since $(L^1(\HH_n)^{U(n)},\ast)$ is a commutative algebra, its spectrum $\Sigma$ is given by the family of the spherical functions $\{\varphi_{\lambda,k}\}_{\lambda\in{\RR\setminus\{0\}},k\in\NN}\cup\{\eta_r\}_{r\in\RR^{\geq0}}$ associated with the Gelfand pair $(\HH_n,U(n))$, see \cite{HR}, \cite{K} and \cite{S}.

As usual, $\mathcal{U}(\mathfrak{h}_{n})$ will denote its universal enveloping algebra, which can be identified with the algebra of left invariant differential operators on $\HH_n$.  It is well known that the commutative subalgebra $\mathcal{U}(\mathfrak{h}_{n})^{U(n)}$ of the elements which commute with the action of $U(n)$ is generated by $S$ and the Heisenberg Laplacian $L$.  The spherical functions are eigenfunction of the operators $L$ and $S$, they satisfying 
\begin{equation}\label{0.1.0}
\left\{
\begin{array}
[c]{l}%
L\varphi_{\lambda,k} = -|\lambda|(2k+n)\varphi_{\lambda,k},\qquad \lambda\in{\RR\setminus\{0\}},k\in\NN\\[10pt]
iS\varphi_{\lambda,k}= \lambda\varphi_{\lambda,k}.
\end{array}
\right.
\end{equation} 
and
\begin{equation*}
\left\{
\begin{array}
[c]{l}%
L\eta_{r} = -r^2\eta_{r},\qquad r\in\RR,r\geq0\\[10pt]
iS\eta_{r}= 0.
\end{array}
\right.
\end{equation*} 
Explicitly
\begin{align*}
&\varphi_{\lambda,k}(z,s)=e^{i\lambda s}L_{k}^{n-1}\left( \frac{|\lambda|}{2} |z|^{2} \right) e^{-\frac{|\lambda|}{4}|z|^{2}},\\
&\eta_r(z,s)=\frac{2^{n-1}(n-1)!}{(r|z|)^{n-1}}J_{n-1}(r|z|),
\end{align*}
where $L_{k}^{n-1}$ denotes, as usual, a Laguerre polynomial of order $n-1$ and degree $k$ normalized by $L_{k}^{n-1}(0)=1$ and $J_{n-1}$ is a Bessel function of order $n-1$ of the first kind. The $\varphi_{\lambda,k}$ functions satisfy the following properties: 
\begin{align}\label{0.1.3}
&\text{If }c\in\rr\text{ then }\varphi_{c\lambda,k}(z,s)=\delta_c\varphi_{\lambda,k}(z,s)=\varphi_{\lambda,k}(\sqrt{c}\,z,cs),\\
\nonumber&\|\varphi_{\lambda,k}\|_{L^\infty(\HH_n)}=1.
\end{align}

The spectrum $\Sigma$ is identify with the set of eigenvalues, $\Sigma=\{(\lambda,|\lambda|(2k+n)):\lambda\in\RR\setminus\{0\},k\in\NN\}\cup\{r\in\RR,r\geq0\}$, with the following measure, if $g\in L^1(\Sigma)$ we have 
$$
\|g\|_{L^1(\Sigma)}=\sum_{k\in\NN}\int_\RR |g(\lambda,k)|\,|\lambda|^nd\lambda<\infty.
$$

For $f\in \mathcal{S}(\HH_n)$ we define the spherical transform, $\widehat{f}:\Sigma\longrightarrow\rr$, by 
\begin{align}\label{0.1.1}
&\widehat{f}(\lambda,k)=\int_{\HH_n}f(z,s)\varphi_{\lambda,k}(-z,-s)\,dz\,ds\\
\nonumber&\widehat{f}(0,r)=\int_{\HH_n}f(z,s)\eta_{r}(-z,-s)\,dz\,ds.
\end{align}

If $f\in L^1(\HH_n)^{U(n)}$ and $\widehat f\in L^1(\Sigma)$, (for example $f\in \mathcal{S}(\HH_n)^{U(n)}$), we use the next Plancherel inversion formula to decompose $f$, see \cite{S}, 
\begin{align}\label{0.1.2}
   f(z,s)&=\sum\limits_{k\ge 0} \int\limits_{-\infty}^{\infty} (f\ast \varphi_{\lambda,k})(z,s) |\lambda|^{n} d\lambda\\
   &=\sum\limits_{k\ge 0} \int\limits_{-\infty}^{\infty} \widehat{f}(\lambda,k)\varphi_{\lambda,k}(z,s)|\lambda|^{n} d\lambda.\nonumber
\end{align}
 
Now let us consider the classical heat equation for the Heisenberg group, defined by
\begin{equation}\label{0.2}
\left\{
\begin{array}
[c]{l}%
v_t(z,s,t)=Lv(z,s,t), \\[10pt]
v(z,s,0)=u_0(z,s).
\end{array}
\right.
\end{equation} 
In \cite{Fo} the author proved there is a unique heat kernel $P:\HH_n\times (0,\infty)\longrightarrow \rr$, $P(z,s,t)=P_t(z,s)$ with $P_0=\delta_0$,  $P_t\geq0$ and $\int_{\HH_n}P_t=1$. The solution of the equation (\ref{0.2}) is given by $v(z,s,t)=P_t\ast u_0(z,s)$, where the convolution product is in the Heisenberg group. He also proves that $P_t$ is $C^\infty$, (see also \cite{ABG}, 
 \cite{C}, \cite{E}, \cite{O}, 
  and \cite{CCG}). 

In \cite{R} the author proves that if $u_0\in L^1(\HH_n)$ then 
\begin{align}\label{0.3}
\|v(\cdot,\cdot,t)\|_{\infty}\leq C t^{\frac{-(2n+2)}2},
\end{align}
 where the constant $C$ depends on the norm $\|u_0\|_{L^1(\HH_n)}$. 


In this work we consider the nonlocal equation given by
\begin{align}\label{pro0}
u_t(z,s,t)=J\ast u(z,s,t)-u(z,s,t),
\end{align}
where the convolution product is in the Heisenberg group and $J$ satisfies the following hypothesis:

\textit{(H)} $J:\HH_n\rightarrow \RR$ is a real function invariant by the action of $U(n)$ with $\displaystyle\int_{\mathbb{H}_{n}}J(z,s)\,dzds=1$. 

We will assume \textit{(H)} throughout   the paper.

Let us now state our results concerning the asymptotic behavior.

The first problem to be addressed is the Cauchy diffusion problem in $\HH_n$. We consider the equation 
\begin{equation}\label{pro1}
\left\{
\begin{array}
[c]{l}%
u_t(z,s,t)=J\ast u(z,s,t)-u(z,s,t), \\[10pt]
u(z,s,0)=u_0(z,s).
\end{array}
\right.
\end{equation} 
 
For this problem we study the asymptotic behavior in the infinity and use the spherical transform to prove the following result    
 
\begin{theorem}\label{1}
Let $u$ the solution of the problem (\ref{pro1}) with $u_0$ in $L^1(\HH_n)^{U(n)}$ and $\widehat{u_0}$ in $L^1(\Sigma)$. Assume that $J$ satisfies \textit{(H)} and that $\widehat{J}(\lambda,k)<1$ for $\lambda\neq0$, $k\in\NN$. Also we assume 
$$\widehat{J}(\lambda,k)=1-|\lambda|(2k+n)+o(|\lambda|(2k+n)),\qquad \text{with}\quad\, \lim_{|\lambda|(2k+n)\rightarrow0}\frac{o(|\lambda|(2k+n))}{|\lambda|(2k+n)}=0.$$
Then the asymptotic behavior of $u(z,s,t)$ is given by 
\begin{align}\label{T1}
\lim_{t\longrightarrow \infty}t^{n+1}\max_{(z,s)}|u(z,s,t)-v(z,s,t)|=0,
\end{align}
where $v$ is the solution of heat equation for the Heisenberg group (\ref{0.2}).\\
The asymptotic profile is given by:  
\begin{align*}
&\lim_{t\longrightarrow \infty}\max_{(z,s)}|t^{n+1}\delta_{t}u(z,s,t)-G_{u_0}(z,s)|=0,
\end{align*}
where $G_{u_0}(z,s)$ satisfies $\widehat{G_{u_0}}(\lambda,k)=e^{-|\lambda|(2k+n)}\widehat{u_0}(0,k)$ and $\delta_t(z,s)=(t^{\frac12}z,ts)$.\\
We also have,  
$$
\|u(\cdot,\cdot,t)\|_{L^\infty(\HH_n)}\leq Ct^{-(n+1)},
$$
and by interpolation for $2<p<\infty$,
$$
\|u(\cdot,\cdot,t)\|_{L^p(\HH_n)}\leq Ct^{-(n+1)(\frac{p-2}{p})}.
$$
\end{theorem} 
\begin{remark}
By \textit{(H)} $\displaystyle\int_{\HH_n}J(z,s)dzds=1$  then if $J(z,s)>0$ for all $(z,s)\in\HH_n$ we have  that $\widehat{J}(\lambda,k)<1$.   
\end{remark}
\begin{remark}
In the literature estimates of the decay in infinite norm have been obtained only for nonlocal equation that   approximate the laplacian operator and not for a more general elliptic operator. The Heisenberg laplacian operator for a function $u$ invariant by the action of $U(n)$, is given in polar coordinates by
$$Lu=\frac{\partial^2u}{\partial r^2}+\frac{2n-2}{r}\frac{\partial u}{\partial r}+\frac{r^2}{4}\frac{\partial^2 u}{\partial s^2}.$$
where
$r^2=\sum_{i=1}^n x^2+y^2$.
For this reason, Theorem \ref{1} gives an example of another elliptic operator that can be approximated by a nonlocal equation in infinite norm. 
\end{remark} 

Let us see the existence of a function $J$ that satisfies the hypotheses of the Theorem \ref{T1}.
 \begin{lemma}
 There exist a real function $J$ invariant by the action of $U(n)$ with $\displaystyle\int_{\mathbb{H}_{n}}J(z,s)\,dzds=1$ and $\widehat{J}(\lambda,k)<1$ for $\lambda\neq0$, $k\in\NN$. Moreover, the spherical transform of $J$ is of the form
$$\widehat{J}(\lambda,k)=1-|\lambda|(2k+n)+o(|\lambda|(2k+n)),\qquad \text{with}\quad\, \lim_{|\lambda|(2k+n)\rightarrow0}\frac{o(|\lambda|(2k+n))}{|\lambda|(2k+n)}=0.$$
 \end{lemma}
 \begin{proof}
  Let $$g(\lambda,k)=e^{-|\lambda|(2k+n)}=\sum\limits_{j\ge 0}\frac{(-|\lambda|(2k+n))^j}{j!}=1-|\lambda|(2k+n)+o(|\lambda|(2k+n)).$$
We have
\begin{align*}
\|g\|_{L^1(\Sigma)}&=\sum_{k\in\NN}\int_\RR e^{-|\lambda|(2k+n)}\,|\lambda|^nd\lambda\\
&=\sum_{k\in\NN}\frac{1}{(2k+n)^{n+1}}\int_\RR e^{-|\xi|}\,|\xi|^nd\xi<\infty.
\end{align*}
We can apply the inverse spherical transform to the function $g\in L^1(\Sigma)$ in order to obtain a kernel $J$ invariant by the action of $U(n)$ with $\displaystyle\int_{\mathbb{H}_{n}}J(z,s)\,dzds=1$,  such that $\widehat{J}(\lambda,k)=g(\lambda,k)=e^{-|\lambda|(2k+n)}$. 
Now we observe 
\begin{align*}
\widehat{\overline {J(\lambda,k)}}&=\int_{\HH_n}\overline {J(z,s)}\varphi_{\lambda,k}(-z,-s)\,dz\,ds\\
&=\overline {\int_{\HH_n}J(z,s)\overline{\varphi_{\lambda,k}(-z,-s)}\,dz\,ds}\\
&=\overline {\int_{\HH_n}J(z,s)\varphi_{-\lambda,k}(-z,-s)\,dz\,ds}\\
&=g(-\lambda,k)=g(\lambda,k)=\widehat{J(\lambda,k)}.
\end{align*}
Then $J$ is a real function and satisfies the Lemma.
\end{proof} 
   
Next we consider a bounded smooth domain $\Omega \subset \HH_n$ and impose boundary conditions to our model. From now on we assume that $J$ is continuous. We consider the next Dirichlet problem
\begin{equation}\label{pro2}
\left\{
\begin{array}
[c]{l}%
u_t(z,s,t)=J\ast u(z,s,t)-u(z,s,t), \qquad\text{for}\,\, (z,s) \in \Omega \,\, \text{and}\,\, t>0,\\[10pt]
u(z,s,t)=g(z,s,t), \,\,\,\,\,\quad\qquad\qquad\quad\qquad\text{for}\,\, (z,s) \notin \Omega \,\, \text{and}\,\, t>0,\\[10pt]
u(z,s,0)=u_0(z,s), \qquad\qquad\qquad\qquad\quad\text{for}\,\, (z,s) \in \Omega.
\end{array}
\right.
\end{equation}
If $J$ satisfies the following hypothesis 

\textit{($\tilde H$)} $J$ is continuous, no negative with $J(0,0)>0$; $J$ have compact support and is symmetric in the variable $s$. We assume there exists a constant $C_1$ with $\displaystyle\int_{\HH_n} J(z,s)x_j^2\,dzds=C_1$, $\displaystyle\int_{\HH_n} J(z,s)y_j^2\,dzds=C_1$, $\displaystyle\int_{\HH_n} J(z,s)s^2\,dzds=C_1$. 

We will consider the rescaled kernel $$J^\epsilon(z,s)=\displaystyle\frac{2C_1^{-1}}{\epsilon^{2n+2}}\delta_{\epsilon^{-2}}J\left( z,s\right)=\displaystyle\frac{2C_1^{-1}}{\epsilon^{2n+2}}J\left(\frac z\epsilon,\frac{s}{\epsilon^2}\right)$$ and the problem 
\begin{equation}\label{pro3}
\left\{
\begin{array}
[c]{l}%
u^\epsilon_t(z,s,t)=\displaystyle\frac1{\epsilon^2}J^\epsilon\ast u(z,s,t)-u(z,s,t), \quad\text{for}\,\, (z,s) \in \Omega \,\, \text{and}\,\, t>0,\\[10pt]
u^\epsilon(z,s,t)=g(z,s,t), \qquad\qquad\qquad\quad\qquad\text{for}\,\, (z,s) \notin \Omega \,\, \text{and}\,\, t>0,\\[10pt]
u^\epsilon(z,s,0)=u_0(z,s), \,\qquad\qquad\qquad\qquad\quad\text{for}\,\, (z,s) \in \Omega.
\end{array}
\right.
\end{equation}

We prove that the solution of \eqref{pro3} approximate uniformly to the solution of the corresponding Dirichlet problem for the classical heat equation, given by
\begin{equation}\label{pro4}
\left\{
\begin{array}
[c]{l}%
v_t(z,s,t)=Lv(z,s,t), \qquad\text{for}\,\, (z,s) \in \Omega \,\, \text{and}\,\, t>0,\\[10pt]
v_t(z,s,t)=g(z,s,t), \,\,\,\,\qquad\text{for}\,\, (z,s) \in \partial \Omega \,\, \text{and}\,\, t>0,\\[10pt]
v(z,s,0)=u_0(z,s),  \,\qquad\quad\text{for}\,\, (z,s) \in \Omega.
\end{array}
\right.
\end{equation}

Our  result are as follows.
\begin{theorem}\label{3}
 Let $\Omega$ be a bounded $C^{2+\alpha}$ domain for some $0< \alpha < 1$. Let $v \in C^{2+\alpha,1+\alpha/2} (\overline\Omega \times [0, T ])$ be the solution to (\ref{pro4}) and let $u^\epsilon$ be the solution to (\ref{pro3}) with $J^\epsilon$ as above and $J$ satisfying \textit{(H)} and \textit{($\tilde H$)}. Then, there exists $C = C(T )$ such that
\begin{align*}
\sup_{t\in[0,T]}\|u^\epsilon(\cdot,\cdot,t)-v(\cdot,\cdot,t)\|_{L^\infty(\Omega)}\leq C\,\epsilon^\alpha, \qquad\qquad\text{as}\qquad \epsilon\rightarrow0. 
\end{align*} 
\end{theorem}
\begin{remark}
Observe that since the initial data $u_0(z,s)$ is not necessarily invariant by the action of $U(n)$, $L$ is given by the formula \eqref{0.0.0.1} and the solution of problem  \eqref{pro3} approaches to the solution of a more irregular equation, given in \eqref{pro4}.
\end{remark} 

Finally we observe, that if $J$ is symmetric in the variable $s$ and as $J$ is invariant  by the action of $U(n)$, we have
\begin{align}\label{s}
J\left(z-\tilde z,s-\tilde s-\frac12\text{Im}\langle z,\tilde z\rangle\right)&=J\left(\tilde z-z,\tilde s-s-(-\frac12\text{Im}\langle z,\tilde z\rangle)\right)\\
&=J\left(\tilde z-z,\tilde s-s-\frac12\text{Im}\langle \tilde z,z\rangle\right).\nonumber
\end{align}
Then, if we write $K((z,s),(\tilde z,\tilde s))=J\left(z-\tilde z,s-\tilde s-\frac12\text{Im}\langle z,\tilde z\rangle\right)$, $K$ is a non-negative and symmetric Kernel. Therefore Theorem 2 of \cite{CCR} is true for the nonlocal equation defined by the kernel $K$. That is to say that $g(z,s,t)\equiv0$ in \eqref{2} and $J$ is also symmetric in the variable $s$, we find an exponential decay given by the first eigenvalue of an associated problem and the asymptotic behavior of solutions is described by the unique (up to a constant) associated eigenfunction. Let $\lambda_1=\lambda_1(\Omega)$ be given by 
\begin{equation}\label{aut}
\lambda_1\!=\!\inf_{u\in L^2(\HH^n)}\dfrac{\displaystyle\!\frac12\!\int_{\HH_n}\!\int_{\HH_n}\!J\left((z-\tilde{z},s-\tilde{s}-\frac12\text{Im}\langle z,\tilde{z}\rangle)\right)\!(u(z,s)-u(\tilde{z},\tilde{z}))^2\,dzds\,d\tilde{z}d\tilde{s}}{\displaystyle\int_\Omega(u(z,s)^2)dzds}.
\end{equation}
\begin{theorem}\label{2}
Let $u_0 \in L^1(\Omega)\cap L^2 (\Omega)$. Assume that $J$ is continuous, satisfies \textit{(H)} and is symmetric in the variable $s$. Then the solutions of \eqref{pro2}, with $g(z,s,t)\equiv0$, decay to zero as $t \rightarrow\infty$ with an exponential rate
\begin{align*}
\|u(\cdot,\cdot,t)\|_{L^2(\Omega)}\leq \|u_0\|_{L^2(\Omega)}e^{-\lambda_1t}.
\end{align*}
If $u_0$ is continuous, positive and bounded then there exist positive constants $C$ and $\tilde{C}$ such
that
\begin{align*}
\|u(\cdot,\cdot,t)\|_{L^\infty(\Omega)}\leq Ce^{-\lambda_1t},
\end{align*}
and
\begin{align*}
\lim_{t\rightarrow0}\max_{(z,s)}|e^{\lambda_1t}u(z,s,t)-\tilde{C}\phi_1(z,t)|=0,
\end{align*}
where $\phi_1$ is the eigenfunction associated to $\lambda_1$. 
\end{theorem}

We consider next the Neumann boundary conditions: 
\begin{equation}\label{pro5}
\left\{
\begin{array}
[c]{l}%
u_t(z,s,t)=\displaystyle\int_{\Omega}J(z-\tilde z,s-\tilde s-\frac12\text{Im}\langle z, \tilde z\rangle)[u(\tilde z, \tilde s,t)-u(z,s,t)]d\tilde zd\tilde s,\\
u(z,s,0)=u_0(z,s).
\end{array}
\right.
\end{equation}
If we impose that  $J$ is symmetric in the variable $s$ by equation \eqref{s} the Theorem 3 of \cite{CCR} is true. And, in this case, we find that the asymptotic behavior is given by an exponential decay determined by an eigenvalue problem. Let $\beta_1$ be given by:
\begin{align}\label{beta}
\beta_1=\inf_{u\in L^2(\Omega),\,\int u=0}\frac{\displaystyle\frac12\int_{\Omega}\int_{\Omega}J(z-\tilde z,s-\tilde s-\frac12\text{Im}\langle z, \tilde z\rangle)[u(\tilde z, \tilde s)-u(z,s)]^2d\tilde zd\tilde sdzds}{\int_\Omega (u(z,s))^2dzds}.
\end{align}
\begin{theorem}\label{4}
Let $J$ be a continuous kernel symmetric in the variable $s$ that satisfies \textit{(H)}. For every $u_0 \in L^1(\Omega)$ there exists a unique solution $u$ of \eqref{pro5} such that $u \in C([0,\infty); L^1 (\Omega))$. This
solution preserves the total mass in $\Omega$:
$$
\int_{\Omega} u(z,s,t)\,dzds=\int_{\Omega} u_0(z,s)\,dzds.
$$
Moreover, let $\mathcal M=\displaystyle\frac{1}{|\Omega|}\int_\Omega u_0(z,s)\,dzds$. Then the asymptotic behavior of solutions of \eqref{pro5} is described as follows: if $u_0 \in L^2(\Omega)$,
\begin{align*}
\|u(\cdot,\cdot,t)-\mathcal M\|_{L^2(\Omega)}\leq e^{-\beta_1t}\|u_0-\mathcal M\|_{L^2(\Omega)},
\end{align*}
and if $u_0$ is continuous and bounded there exists a positive constant $C$ such that
\begin{align*}
\|u(\cdot,\cdot,t)-\mathcal M\|_{L^\infty(\Omega)}\leq Ce^{-\beta_1t}.
\end{align*}
\end{theorem}

The rest of the paper is organized as follows: in Section \ref{sect2}, we prove existence and
uniqueness of the Cauchy problem given by (\ref{pro1}) and we also prove Theorem \ref{1}. In Section \ref{sect3} we prove existence, uniqueness and a comparison principle of the Dirichlet problem given by (\ref{pro2}), and we also prove the convergence result for Dirichlet problem, Theorem \ref{3}.

\section{The Cauchy problem}

\label{sect2}

In this section, we will use that the function $u_0$ is invariant by the action of $U(n)$, this allows us to use the spherical transform of $\HH_n$ in order to obtain explicit solutions of Cauchy problem \eqref{pro1}. 


\begin{theorem}\label{2.0}
Let $u_0$ in $L^1(\HH_n)^{U(n)}$ and $\widehat{u_0}$ in $L^1(\Sigma)$. Let $J$ satisfy \textit{(H)}. Then there exists a unique solution $u\in C^0([0,\infty),L^1(\HH_n))$ of problem \eqref{pro1} and it is given by:
$$
\widehat{u}(\lambda,k,t)=e^{(\widehat{J}(\lambda,k)-1)t}\widehat{u_0}(\lambda,k).
$$
\end{theorem}
\begin{proof}
First observe that since $\displaystyle\int_{\mathbb{H}_{n}}J(z,s)\,dzds=1$, then $\widehat{J}\in C_0(\Sigma)$ and $\widehat{J}(0,0)=1$.

We have
$$u_t(z,s,t)=J\ast u(z,s,t)-u(z,s,t).$$
Applying the spherical transform to this equation, we obtain:
$$\widehat{u}_t(\lambda,k,t)=(\widehat{J}(\lambda,k)-1)\widehat{u}(\lambda,k,t).$$
Hence,
$$\widehat{u}(\lambda,k,t)=e^{(\widehat{J}(\lambda,k)-1)t}\widehat{u}_0(\lambda,k).$$
Since $\widehat{u_0} \in L^1 (\Sigma )$ and $e^{(\widehat{J}-1)t}$ is continuous and bounded, the result follows by taking the inverse of the spherical transform.
\end{proof}

\begin{lemma}
Let $J\in \mathcal{S}(\HH_n)^{U(n)}$ satisfy \textit{(H)} and $u_0=\delta_0$ (the Dirac delta in $\HH_n$). Then the fundamental solution of \eqref{pro1} can be decomposed as
$$
w(z,s,t)=e^{-t}\delta_0+\nu(z,s,t),
$$
with $\nu(z,s,t)$ smooth. Moreover, if $u$ is a solution of \eqref{pro1} with initial condition a function $u_0$ invariant by the action of $U(n)$, it can be written as
$$
u(z,s,t)=w\ast u_0(z,s,t).
$$
\end{lemma}
\begin{proof}
By the previous result, we have
$$\widehat{w}_t(\lambda,k,t)=(\widehat{J}(\lambda,k)-1)\widehat{w}(\lambda,k,t).$$
Hence $\widehat{\delta_0}=1$, in the sense of distributions, we have
$$\widehat{w}(\lambda,k,t)=e^{(\widehat{J}(\lambda,k)-1)t}=e^{-t}(e^{\widehat{J}(\lambda,k)t}-1)+e^{-t}.$$
Now let us prove that, for each fixed $t$, $e^{\widehat{J}(\lambda,k)t}-1\in L^1(\Sigma)$. By the mean value theorem
\begin{align*}
\sum_{k\in\NN}\int_\RR |e^{\widehat{J}(\lambda,k)t}-1|\,|\lambda|^nd\lambda&=\sum_{k\in\NN}\int_\RR |\widehat{J}(\lambda,k)te^{\widehat{J}(\lambda',k')t}|\,|\lambda|^nd\lambda\\
&\leq C\sum_{k\in\NN}\int_\RR |\widehat{J}(\lambda,k)t|\,|\lambda|^nd\lambda
\end{align*}
By \cite{ABR}  exist a function $g\in S(\RR^2)$ such that $\widehat{J}(\lambda,k)=g(\lambda,|\lambda|(2k+n))$ then
\begin{align*}
\sum_{k\in\NN}\int_\RR |e^{\widehat{J}(\lambda,k)t}-1|\,|\lambda|^nd\lambda&\leq C\sum_{k\in\NN}\int_\RR |g(\lambda,|\lambda|(2k+n)t)|\,|\lambda|^nd\lambda\\
&=C \frac1{t^{n+1}}\sum_{k\in\NN}\int_\RR \left|g\left(\frac{\eta}{(2k+n)t},\eta\right)\right|\,\frac{|\eta|^n}{(2k+n)^{n+1}}d\eta
\end{align*}
As $g$ belongs in  $S(\RR^2)$ exist constants $C_1$ and $C_2$ such that $|g(x,y)|\leq C_1$ and $|g(x,y)|\leq \frac{C_2}{|y|^{n+2}}$, then
\begin{align*}
\sum_{k\in\NN}&\int_\RR |e^{\widehat{J}(\lambda,k)t}-1|\,|\lambda|^nd\lambda\\
&\leq C\frac1{t^{n+1}}\sum_{k\in\NN} \frac{1}{(2k+n)^{n+1}}\left[\int_{|\eta|\leq1} C_1\,|\eta|^{n}d\eta  + \int_{|\eta|>1} \frac{C_2}{|\eta|^{2}}\,d\eta       \right]<\infty
\end{align*}
Therefore the first part of the lemma follows applying the inverse spherical transform. 

Note that since $J$ and $u_0$ are invariant by the action of $U(n)$ it is enough to show that there exist $L^{(r)}(\nu)$ and $S^{r}(\nu)$, for all $r\in\NN$, to prove that  $\nu\in C^{\infty}(\HH_n)^{U(n)}$. this is shown similarly to the previous account using \eqref{0.1.0}.

To finish the proof, we observe, that 
$$\widehat{w\ast u_0}(\lambda,k,t)=\widehat{w}(\lambda,k,t)\widehat{u_0}(\lambda,k)=e^{(\widehat{J}(\lambda,k)-1)t}\widehat{u_0}(\lambda,k).$$
By Theorem \eqref{2.0} the solution of problem \eqref{pro1} satisfies 
$$\widehat{u}(\lambda,k,t)=e^{(\widehat{J}(\lambda,k)-1)t}\widehat{u}_0(\lambda,k).$$
Then the result is followed since the spherical transform is injective.
\end{proof}

Next we will prove the asymptotic behavior for the nonlocal diffusion equation \eqref{pro1}.
\begin{proof}[Proof of Theorem \ref{1}]
We remark that from our hypotheses on $J$, 
\begin{align*}
&\widehat{J}(\lambda,k)=1-|\lambda|(2k+n)+o(|\lambda|(2k+n)),\quad \text{with }\,\lim_{|\lambda|(2k+n)\longrightarrow0}\frac{o(|\lambda|(2k+n))}{|\lambda|(2k+n)}=0 .
\end{align*}
We have that,
\begin{align}\label{q1}
\hat{J}(\lambda,k)\leq 1-|\lambda|(2k+n)+|\lambda|(2k+n)h(|\lambda|(2k+n)),
\end{align}
where $h$ is a bounded positive function and $\lim_{|\lambda|(2k+n)\longrightarrow 0}h(|\lambda|(2k+n))=0.$ We recall that $|\hat{J}(\lambda,k)|\leq1$, then there exist a number $\xi>0$ and constants $D>0$ and $E>0$ such that 
\begin{align}\label{q2}
&\hat{J}(\lambda,k)\leq 1-D|\lambda|(2k+n),\qquad\text{if}\qquad|\lambda|(2k+n)\leq \xi,\\
&\quad\hat{J}(\lambda,k)\leq 1-E,\qquad\qquad\qquad\text{if}\qquad|\lambda|(2k+n)> \xi.
\end{align}

As in the proof of the  Theorem \eqref{2.0}, we have
$$
\widehat{u}(\lambda,k,t)=e^{(\widehat{J}(\lambda,k)-1)t}\widehat{u}_0(\lambda,k).
$$
On the other hand, let $v(z,s, t)$ be a solution of the heat Heisenberg equation, with the same initial datum $v(z,s,0)=u_0(z,s)$. Taking the spherical transform and by equations \eqref{0.1.0} and \eqref{0.2} we get
$$
\widehat{v}(\lambda,k,t)=e^{-|\lambda|(2k+n)t}\widehat{u}_0(\lambda,k).
$$
Then, by \eqref{0.1.2} and \eqref{0.1.3}, we have 
\begin{align*}
|u(z,s,t)-v(z,s,t)|&=\left|\sum_{k\geq0}\int_{-\infty}^{\infty} (\widehat{u}-\widehat{v})(\lambda,k,t)\varphi_{\lambda,k}(z,s)|\lambda|^{n} d\lambda\right|\\
                   &\leq\sum_{k\geq0}\int_{-\infty}^{\infty} \left|\left(e^{(\widehat{J}(\lambda,k)-1)t}-e^{-|\lambda|(2k+n)t}\right)\widehat{u}_0(\lambda,k)\right||\lambda|^{n} d\lambda.              
                    \end{align*}
We decompose the equation in two parts, when $|\lambda|(2k+n)\sqrt{t}\geq 1$ and $|\lambda|(2k+n)\sqrt{t}< 1$.
\begin{align*}
&|u(z,s,t)\!-\!v(z,s,t)|\!\leq\! \sum_{k\geq 0}\int_{|\lambda|\geq\frac{1}{(2k+n)\sqrt{t}}}\! \left|\left(e^{(\widehat{J}(\lambda,k)-1)t}-e^{-|\lambda|(2k+n)t}\right)\widehat{u}_0(\lambda,k)\right|\!|\lambda|^{n} d\lambda\\
                   &\qquad\qquad + \sum_{k\geq 0}\int_{|\lambda|< \frac{1}{(2k+n)\sqrt{t}}} \left|\left(e^{(\widehat{J}(\lambda,k)-1)t}-e^{-|\lambda|(2k+n)t}\right)\widehat{u}_0(\lambda,k)\right||\lambda|^{n} d\lambda\\
                   &\qquad:=I+II.
\end{align*}
First we work with $I$, 
\begin{align*}
I&\leq\sum_{k\geq 0}\int_{|\lambda|>\frac{1}{(2k+n)\sqrt{t}}} \left|\left(e^{(\widehat{J}(\lambda,k)-1)t}-e^{-|\lambda|(2k+n)t}\right)\widehat{u}_0(\lambda,k)\right||\lambda|^{n} d\lambda\\
 &\leq \sum_{k\geq 0}\int_{|\lambda|>\frac{1}{(2k+n)\sqrt{t}}} \left|e^{(\widehat{J}(\lambda,k)-1)t}\widehat{u}_0(k,\lambda)\right||\lambda|^{n} d\lambda\\
 &\qquad+ \sum_{k\geq 0}\int_{|\lambda|>\frac{1}{(2k+n)\sqrt{t}}} \left|e^{-|\lambda|(2k+n)t}\widehat{u}_0(\lambda,k)\right||\lambda|^{n} d\lambda\\
 &:=I_1+I_2.
\end{align*}
For $I_2$, we make the change of variables $\lambda(2k+n) t=\eta$, then $|\lambda|(2k+n) t=|\eta|$ and $d\lambda(2k+n) t=d\eta$, and
\begin{align*}
t^{n+1}I_2&=t^{n+1}\sum_{k\geq 0}\int_{|\lambda|>\frac{1}{(2k+n)\sqrt{t}}} \left|e^{-|\lambda|(2k+n)t}\widehat{u}_0(\lambda,k)\right||\lambda|^{n} d\lambda\\
&\leq\|\widehat{u}_0\|_\infty\left[\int_{|\eta|>\sqrt{t}} e^{-|\eta|}|\eta|^{n} d\eta\right]\sum_{k\geq 0}\frac{1}{(2k+n)^{n+1}}.
\end{align*} 
Note that the sum is finite and by the dominated convergence theorem, $$\lim_{t\rightarrow\infty}t^{n+1}I_2=0.$$

Now, we work with $I_1$. By \eqref{q2} $I_1$ is bounded by
\begin{align*}
I_1&=\sum_{k\geq 0}\int_{|\lambda|>\frac{1}{(2k+n)\sqrt{t}}} \left|e^{(\hat{J}(\lambda,k)-1)t}\hat{u}_0(\lambda,k)\right||\lambda|^{n} d\lambda\\
&=\sum_{k\geq 0}\int_{\frac{\xi}{2k+n}>|\lambda|>\frac{1}{(2k+n)\sqrt{t}}} \left|e^{(\hat{J}(\lambda,k)-1)t}\hat{u}_0(\lambda,k)\right||\lambda|^{n} d\lambda\\
&\qquad+\sum_{k\geq 0}\int_{|\lambda|\geq\frac{\xi}{2k+n}} \left|e^{(\hat{J}(\lambda,k)-1)t}\hat{u}_0(\lambda,k)\right||\lambda|^{n} d\lambda\\
&\leq \sum_{k\geq 0}\int_{\frac{\xi}{2k+n}>|\lambda|>\frac{1}{(2k+n)\sqrt{t}}} \left|e^{-D|\lambda|(2k+n)t}\hat{u}_0(\lambda,k)\right||\lambda|^{n} d\lambda +\|\hat{u}_0\|_{L^1(\Sigma)}e^{-E t}.
\end{align*}
We now make the change of variables $\lambda(2k+n) t=\eta$, and then 
\begin{align*}
 t^{n+1}I_1&\leq \|\hat{u}_0\|_\infty \sum_{k\geq 0}\frac{1}{(2k+n)^{n+1}}\int_{\xi t>|\eta|>\sqrt{t}} e^{-D|\eta|}|\eta|^{n} d\eta +\|\hat{u}_0\|_{L^1(\Sigma)}t^{n+1}e^{-E t}\\
 &\leq \|\hat{u}_0\|_\infty \left[\int_{|\eta|>\sqrt{t}} e^{-D|\eta|}|\eta|^{n} d\eta\right]\sum_{k\geq 0}\frac{1}{(2k+n)^{n+1}} +\|\hat{u}_0\|_{L^1(\Sigma)}t^{n+1}e^{-E t}.
\end{align*}
Therefore $t^{n+1}I_1\rightarrow0$ when $t\rightarrow\infty.$

Finally we will estimate $II$. Again we make the change of variables  $\lambda(2k+n) t=\eta$, if $\sqrt{t}$ is a sufficiently large number, by \eqref{q1}, we have  
\begin{align*}
 t^{n+1}II&= t^{n+1} \sum_{k\geq 0}\int_{|\lambda|<\frac{1}{(2k+n)\sqrt{t}}} \left|\left(e^{[\hat{J}-1+|\lambda|(2k+n)]t}-1\right)e^{-|\lambda|(2k+n)t}\hat{u}_0(k,\lambda)\right||\lambda|^{n} d\lambda\\
 &\leq t^{n+1}\|\hat{u}_0\|_\infty \sum_{k\geq 0}\int_{|\lambda|< \frac{1}{(2k+n)\sqrt{t}}} \left|e^{|\lambda|(2k+n)h(|\lambda|(2k+n))t}-1\right|e^{-|\lambda|(2k+n)t}|\lambda|^{n} d\lambda\\
&\leq Ct^{n+1}\|\hat{u}_0\|_\infty\!  \sum_{k\geq 0}\int_{|\lambda|< \frac{1}{(2k+n)\sqrt{t}}}\! |\lambda|(2k+n)h(|\lambda|(2k+n))te^{-|\lambda|(2k+n)t}|\lambda|^{n} d\lambda\\
&\leq C\|\hat{u}_0\|_\infty  \sum_{k\geq 0}\int_{|\eta|< \sqrt{t}} \frac{1}{(2k+n)^{n+1}}h\left(\frac{|\eta|}{t}\right)e^{-|\eta|}|\eta|^{n+1} d\eta\\
&\leq C\|\hat{u}_0\|_\infty  \left(\sum_{k\geq 0}\frac{1}{(2k+n)^{n+1}}\right)\int_{\RR} h\left(\frac{|\eta|}{t}\right)e^{-|\eta|}|\eta|^{n+1} d\eta.
\end{align*}
Observe that $h(\frac{|\eta|}{t})\rightarrow0$ when $t\rightarrow \infty$. Also $h\left(\frac{|\eta|}{t}\right)e^{-|\eta|}|\eta|^{n+1}\leq \|h\|_\infty|\eta|^{n+1}e^{-|\eta|}$, and then by convergence dominated theorem $t^{n+1}II\rightarrow0$ when $t\rightarrow\infty$.


Thus we have showed that 
$$
\lim_{t\longrightarrow \infty}t^{n+1}\max_{(z,s)}|u(z,s,t)-v(z,s,t)|=0
$$
since
\begin{align*}
&\lim_{t\longrightarrow \infty}t^{n+1}\max_{(z,s)}|u(z,s,t)-v(z,s,t)|\\
&\leq\lim_{t\longrightarrow \infty}t^{n+1}\sum_{k\geq0}\int_{-\infty}^{\infty} |\hat{u}-\hat{v}|(\lambda,k,t)|\lambda|^{n} d\lambda=0.
\end{align*}


Now we will prove that the asymptotic profile is given by
\begin{align*}
\lim_{t\rightarrow \infty}\max_{(z,s)}|t^{n+1}u(t^{\frac12}z,ts,t)-G_{u_0}(z,s)|=0,
\end{align*}
where  $G_{u_0}(z,s)$ is the function such that $\widehat{G_{u_0}}(\lambda,k)=e^{-|\lambda|(2k+n)}\widehat{u_0}(0,k)$.

Indeed, we have
\begin{align}\label{qw}
\widehat{v}(t^{-1}\lambda,k,t)=e^{-|\lambda|(2k+n)}\widehat{u_0}(t^{-1}\lambda,k)\rightarrow e^{-|\lambda|(2k+n)}\widehat{u_0}(0,k).
\end{align}
Now, taking the spherical transform and by \eqref{0.1.3} and \eqref{0.1.1}, we get 
\begin{align}\label{we}
t^{n+1}\widehat{\delta_{t}v(\cdot,\cdot,t)}(\lambda,k,t)&=t^{n+1}\int_{\HH_n}v(t^{\frac12}z,ts,t)\varphi_{tt^{-1}\lambda,k}(-z,-s)\,dz\,ds\\ 
\nonumber&=t^{n+1}\int_{\HH_n}v(t^{\frac12}z,ts,t)\varphi_{t^{-1}\lambda,k}(-t^{\frac12}z,-ts)\,dz\,ds\\
\nonumber&=\int_{\HH_n}v(z,s,t)\varphi_{t^{-1}\lambda,k}(-z,-s)\,dz\,ds\\
\nonumber&=\widehat{v}(t^{-1}\lambda,k,t).
\end{align} 
By \eqref{T1}, \eqref{qw} and \eqref{we} we have
\begin{align*}
&\lim_{t\rightarrow \infty}\max_{(z,s)}|t^{n+1}\delta_{t}u(z,s,t)-G_{u_0}(z,s)|\leq\lim_{t\rightarrow \infty}\max_{(z,s)}|t^{n+1}\delta_{t}u(z,s,t)-t^{n+1}\delta_{t}v(z,s,t)|\\
&\qquad\qquad\quad+\lim_{t\rightarrow \infty}\max_{(z,s)}|t^{n+1}\delta_{t}v(z,s,t)-G_{u_0}(z,s)|=0.
\end{align*}

Finally, since $\|v(\cdot,\cdot,t)\|_{\infty}\leq C t^{-(n+1)}$, (see \eqref{0.3}), we have 
$$
\|u(\cdot,\cdot,t)\|_{L^\infty(\HH_n)}\leq Ct^{-(n+1)}.
$$
and by interpolation for $2<p<\infty$,
$$
\|u(\cdot,\cdot,t)\|_{L^p(\HH_n)}\leq \|u(\cdot,\cdot,t)\|_{L^2(\HH_n)}^{\frac2p}\|u(\cdot,\cdot,t)\|_{L^\infty(\HH_n)}^{\frac{p-2}{p}}.
$$
As \eqref{pro1} preserves the $L^2(\HH_n)$ norm, because it is the solution given through the spherical transform, we have
$$
\|u(\cdot,\cdot,t)\|_{L^p(\HH_n)}\leq C t^{-(n+1)(\frac{p-2}{p})}.
$$

\end{proof}

\section{The Dirichlet problem}
\label{sect3}

\subsection{Existence and properties of solutions}\label{sect3.1}

\

We shall first derive the existence and uniqueness of
solutions of \eqref{pro2}, which is a consequence of Banach's fixed point theorem. The main arguments are basically the same of \cite{CCR} or \cite{CER}, but we write them here to make the paper self-contained. 
\begin{theorem}\label{31}
Let $u_0\in L^1(\Omega)$ and be $J$ a kernel that verifies \textit{(H)} and \textit{($\tilde H$)}. Then there exists a unique solution $u$ of \eqref{pro2} such that $u\in C([0,\infty),L^1(\Omega))$.
\end{theorem}

Recall that a solution of the Dirichlet problem is defined as a $u\in C([0,\infty),L^1(\Omega))$  satisfying \eqref{pro2}.
\begin{proof}
We use the Banach's fixed point theorem. Fix $t_0 > 0$ and consider the Banach space
\begin{align*}
X_{t_0}:=\left\{w\in C([0,t_0]; L^1 (\Omega)), \text{and }w(z,s,t)=g(z,s,t)\text{ if }(z,s) \notin \Omega\right\},
\end{align*}
with the norm
\begin{align*}
|||w|||:=\max_{0\leq t\leq t_0}\|w(\cdot,\cdot,t)\|_{L^1(\Omega)}.
\end{align*}
We will obtain the solution as a fixed point of the operator $\mathfrak{T} : X_{ t_ 0} \rightarrow X_{ t_ 0}$ defined by
\begin{equation*}
\mathfrak{T} (w)(z,s,t):=\left\{
\begin{array}
[c]{l}%
w_0(z,s)+\displaystyle\int_{0}^{t}J\ast w(z,s,r)-w(z,s,r)\,dr \qquad\,\,\,\text{if}\,\, (z,s) \in \Omega,\\[10pt]
g(z,s,t) \,\,\,\qquad\qquad\qquad\qquad\qquad\qquad\qquad\qquad\text{if}\,\, (z,s) \notin \Omega,\\[10pt]
\end{array}
\right.
\end{equation*} 
where $w_0(z,s)=w(z,s,0)$.

Let $w,\, v \in X_{t_ 0}$. Then there exists a constant $C$ depending on $J$ and $\Omega$ such that
\begin{align}\label{3.1}
|||\mathfrak{T}(w)-\mathfrak{T}(v)|||\leq Ct_0|||w-v|||+\|w_0-v_0\|_{L^1(\Omega)}.
\end{align}
We will prove \eqref{3.1}. Indeed, 
\begin{align*}
&\int_{\Omega}\left|\mathfrak{T}(w)-\mathfrak{T}(v)\right|(z,s,t)dzds\leq\int_{\Omega}|w_0-v_0|(z,s)dzds\\
&\qquad \quad +\int_{\Omega}\left|\int_{0}^{t}J\ast (w-v)(z,s,r)-(w-v)(z,s,r)\,dr\right|dzds\\
&\qquad \leq\|w_0-v_0\|_{L^1(\Omega)}+tc(\|J\|_{L^\infty(\Omega)}+|\Omega|)\|(w-v)(\cdot,\cdot,t)\|_{L^1(\Omega)}.
\end{align*}
Taking the maximum in $t$  \eqref{3.1} follows.

Now, taking $v_0\equiv v\equiv 0$ in \eqref{3.1} we get that $\mathfrak{T}(w) \in C([0,t_0]; L^1 (\Omega))$
and this says that $\mathfrak{T}$ maps $X_{t_0}$ into $X_{t_0}$.  

Finally, we will consider $X_{t_0,u_0}=\{u\in X_{t_0}:\,u(z,s,0)=u_0(z,s)\}$. $\mathfrak{T}$ maps $X_{t_0,u_0}$ into $X_{t_0,u_0}$ and taking $t_ 0$ such that $(C+1)t_ 0 < 1$, where $C$ is the constant given in \eqref{3.1} we can apply the Banach's fixed point theorem in the interval $[0,t_0]$ because  $ \mathfrak{T}$ is a strict contraction in $X_{ t_ 0,u_0}$. From this we get the existence and uniqueness of the solution in $[0,t_0]$. To extend the solution to $[0,\infty)$ we may take as initial data $u(x, t_ 0 ) \in L^1 (\Omega)$ and obtain a solution up to $[0, 2t_ 0 ]$. Iterating this procedure we get a solution defined in $[0,\infty)$. 
\end{proof}

In order to prove a comparison principle of problem given by (\ref{pro2}) we need to introduce the definition of sub and super solutions.
\begin{definition}
A function $u \in C([0, T ]; L^1 (\Omega))$ is a supersolution of \eqref{pro2} if
\begin{equation}\label{pro2sub}
\left\{
\begin{array}
[c]{l}%
u_t(z,s,t)\geq J\ast u(z,s,t)-u(z,s,t), \qquad\text{for}\,\, (z,s) \in \Omega \,\, \text{and}\,\, t>0,\\[10pt]
u_t(z,s,t)\geq g(z,t), \qquad\qquad\qquad\quad\qquad\text{for}\,\, (z,s) \notin \Omega \,\, \text{and}\,\, t>0,\\[10pt]
u(z,s,0)\geq u_0(z,s), \,\,\,\qquad\qquad\qquad\quad\quad\text{for}\,\, (z,s) \in \Omega.
\end{array}
\right.
\end{equation}
As usual, subsolutions are defined analogously by reversing the inequalities.
\end{definition}
\begin{lemma}\label{com}
Let $u_0 \in C(\overline\Omega)$, $u_0 \geq 0$, and $u \in C(\overline\Omega\times[0, T ])$ a supersolution of
\eqref{pro2} with $g \geq 0$. Then, $u \geq 0$.
\end{lemma}
\begin{proof}
Assume to the contrary that $u(z,s, t)$ is negative in some point. Let $v(z,s, t)\\ = u(z,s, t) + \epsilon t$ with $\epsilon>0$ small such that $v$ is still negative somewhere. Then, if $(z_0,s_0 , t_ 0 )$ is a point where $v$ attains its negative minimum, there it holds that $t_ 0 > 0$ and
\begin{align*}
v_t(z_0,s_0 , t_ 0 )&=u_t(z_0,s_0 , t_ 0 )+\epsilon>J\ast u(z_0,s_0 , t_ 0) -u(z_0,s_0 , t_ 0 )\\
                    &=\int_{\HH_n}J(\tilde z-z,\tilde s-s-\frac12\text{Im}\langle\tilde z , z\rangle)(v(\tilde z,\tilde s,t_0)-v(z_0,s_0 , t_ 0 ))d\tilde zd\tilde s.
\end{align*}
This contradicts that $(z_0,s_0 , t_ 0 )$ is a minimum of $v$. Thus, $u \geq 0$.
\end{proof}
\begin{corollary}\label{comp}
Let $J\in L^\infty (\HH_n)$. Let $u_0$ and $v_0$ in $L^1(\Omega)$ with $u_0 \geq v_ 0$ and
$g, h \in L^\infty  ((0, T ); L^1 (\HH_n \setminus\Omega))$ with $g \geq h$. Let $u$ be a solution of \eqref{pro2} with $u(z,s, 0) = u _0(z,s)$ and Dirichlet datum $g$, and let $v$ be a solution of \eqref{pro2} with $v(z,s, 0) = v_ 0(z,s)$
and datum $h$. Then, $u \geq v$ a.e. $\Omega$.
\end{corollary}
\begin{proof}
Let $w = u - v$. Then, $w$ is a supersolution with initial datum $u_ 0 - v_ 0 \geq 0$ and datum $g - h \geq 0$. Using the continuity of the solutions with respect to the data and the fact that $J \in L^\infty (\HH_n)$, we may assume that $u, v \in C(\Omega\times [0, T ])$. By Lemma \eqref{com} we obtain that $w = u - v \geq 0$. So the corollary is proved.
\end{proof}
\begin{corollary}\label{comp2}
Let $u \in C(\Omega\times [0, T ])$ (resp., $v$) be a supersolution (resp.,
subsolution) of \eqref{pro2}. Then, $u \geq v$.
\end{corollary}
\begin{proof}
It follows from the proof of the previous corollary.
\end{proof}

\subsection{Convergence to the heat equation}
\label{sect3.3}

\

In order to prove a to prove Theorem \ref{3}, let $\tilde v$ be a $C^{ 2+\alpha,1+\alpha/2}$ extension of $v$ to $\HH_n \times [0, T ]$, where $v$ is the solution of \eqref{pro4}.
Let us define the operator
$$
\tilde L_\epsilon(w)(z,s,t):=\frac1{\epsilon^2} J_\epsilon\ast w(z,s,t)-w(z,s,t).
$$
Then $\tilde v$ verifies 
\begin{equation}\label{pro2.1}
\left\{
\begin{array}
[c]{l}%
\tilde v_t(z,s,t)=\tilde L_\epsilon(v)(z,s,t)+F_\epsilon (z,s,t), \qquad\text{for}\,\, (z,s) \in \Omega \,\, \text{and}\,\, t\in[(0,T],\\[10pt]
\tilde v(z,s,t)=g(z,s,t)+G(z,s,t), \quad\,\,\quad\qquad\text{for}\,\, (z,s) \notin \Omega \,\, \text{and}\,\, t\in,(0,T]\\[10pt]
\tilde v(z,s,0)=u_0(z,s), \,\,\,\quad\qquad\qquad\qquad\qquad\text{for}\,\, (z,s) \in \Omega.
\end{array}
\right.
\end{equation}
Since $Lv(z,s,t)=L\tilde v(z,s,t)$ for $(z,s)\in \Omega$, we have
$$
F_\epsilon(z,s,t)=L\tilde v(z,s,t)-\tilde L_\epsilon(v)(z,s,t).
$$
Moreover, as $G$ is smooth and $G(z,s, t) = 0$ if $(z,s)\in \partial\Omega$ we have
$$
G(z,s, t) = o(\epsilon) \qquad \text{for }(z,s)\, \text{ such that dist}((z,s), \partial\Omega)\leq\epsilon d .
$$
We set $w^\epsilon =\tilde v - u^ \epsilon$ and we note that
\begin{equation}\label{pro2.2}
\left\{
\begin{array}
[c]{l}%
w^\epsilon_t(z,s,t)=\tilde L_\epsilon(w^\epsilon)(z,s,t)+F_\epsilon (z,s,t), \quad\text{for}\,\, (z,s) \in \Omega \,\, \text{and}\,\, t\in(0,T],\\[10pt]
w^\epsilon(z,s,t)=G(z,s,t), \,\,\quad\qquad\qquad\quad\qquad\text{for}\,\, (z,s) \notin \Omega \,\, \text{and}\,\, t\in,(0,T],\\[10pt]
w^\epsilon(z,s,0)=0, \qquad\qquad\qquad\,\,\,\qquad\qquad\quad\text{for}\,\, (z,s) \in \Omega.
\end{array}
\right.
\end{equation}
\begin{lemma}\label{lem4}
Let $\tilde v$, $\tilde L_\epsilon$ and $F_\epsilon$ be as previously defined. Then we have that
\begin{align}\label{4.1}
\sup_{t\in[0,T]}\|F_\epsilon\|_{L^\infty(\Omega)}=o(\epsilon^\alpha).
\end{align}
\end{lemma}
\begin{proof}
By  $\tilde v \in C^{ 2+\alpha,1+\alpha/2} (\HH_n\times[0, T])$,  we have that
\begin{align*}
&F_\epsilon(z,s,t)=L\tilde v(z,s,t)-\tilde L_\epsilon(v)(z,s,t)
\end{align*}
In the global coordinate system $(x,y,s)$, we obtain  
\begin{align*}
\tilde L_\epsilon(v)(x,y,s,t)&=\frac{2C_1^{-1}}{\epsilon^{2n+4}}\int_{\RR^{2n+1}}J\left(\frac{\tilde x-{x}}{\epsilon},\frac{ \tilde y-{y}}{\epsilon},\frac{\tilde s-{s}}{\epsilon^2}-\frac12\frac{\tilde x{y}-\tilde y{x}}{\epsilon^2}\right)\\
&\qquad\qquad\cdot( v(\tilde x,\tilde y,\tilde s,t)- v({x},{y},{s},t))\,d\tilde xd\tilde yd\tilde s.
\end{align*}
We now make the change of variables $\frac{\tilde x-{x}}{\epsilon}=\hat x$, $\frac{ \tilde y-{y}}{\epsilon}=\hat y$ and $\frac{\tilde s-{s}}{\epsilon^2}=\hat x$, and so,
\begin{align*}
\tilde L_\epsilon(v)(x,y,s,t)&=\frac{2C_1^{-1}}{\epsilon^{2}}\int_{\RR^{2n+1}}J\left(\hat x,\hat y,\hat s-\frac12\frac{(\epsilon \hat x+x){y}-(\epsilon\hat y+y){x}}{\epsilon^2}\right)\\
&\qquad\qquad\cdot(v(\epsilon \hat x+x,\epsilon \hat y+y,\epsilon \hat s+s,t)- v({x},{y},{s},t))\,d\hat xd\hat yd\hat s\\
&=\frac{2C_1^{-1}}{\epsilon^{2}}\int_{\RR^{2n+1}}J\left(\hat x,\hat y,\hat s-\frac{\epsilon \hat x{y}-\epsilon\hat y{x}}{2\epsilon^2}\right)\\
&\qquad\qquad\cdot( v(\epsilon \hat x+x,\epsilon \hat y+y,\epsilon \hat s+s,t)- v({x},{y},{s},t))\,d\hat xd\hat yd\hat s.
\end{align*}
By a simple Taylor expansion we have
\begin{align*}
&v(\epsilon \hat x+x,\epsilon \hat y+y,\epsilon \hat s+s,t)- v({x},{y},{s},t)\\
&\quad=\sum_{j=1}^n \frac{\partial}{\partial x_j}v({x},{y},{s},t)\epsilon\hat{x}_j+\sum_{j=1}^n \frac{\partial}{\partial y_j}v({x},{y},{s},t)\epsilon\hat{y}_j\\
&\qquad+\frac{\partial}{\partial s}v({x},{y},{s},t)\epsilon^2\hat{s} +\frac{1}{2}\sum_{j,i} \frac{\partial^2}{\partial x_j\partial x_i}v({x},{y},{s},t)\epsilon^{2}\hat{x}_j\hat{x}_i\\
&\qquad+\frac{1}{2}\sum_{j,i} \frac{\partial^2}{\partial y_j\partial y_i} v({x},{y},{s},t)\epsilon^{2}\hat{y}_j\hat{y}_i+ \sum_{j,i} \frac{\partial^2}{\partial x_j\partial y_i} v({x},{y},{s},t)\epsilon^{2}\hat{x}_j\hat{y}_i\\
&\qquad
  +\frac12\sum_{j=1}^n \frac{\partial^2}{\partial x_j\partial s}v({x},{y},{s},t)\epsilon^{3}\hat{x}_j\hat{s}+ \sum_{j=1}^n \frac{\partial^2}{\partial y_j\partial s}v({x},{y},{s},t)\epsilon^{3}\hat{y}_j\hat{s}\\
  &\qquad+ \frac12 \frac{\partial^2}{\partial s^2} v({x},{y},{s},t) \epsilon^{4}\hat{s}^2+o(\epsilon^{2+\alpha}). 
\end{align*}
By the fact that $J$ verifies the hypothesis $(\tilde H)$,
\begin{align*}
&\tilde L_\epsilon(v)(x,y,s,t)=\left[\sum\limits_{j=1}^{n} \frac{\partial^2}{\partial x_{j}^{2}}+\frac{\partial^2}{\partial y_{j}^{2}}+\epsilon^{2}\frac{\partial^2}{\partial s^2}\right] v({x},{y},{s},t)\\ 
&\quad+\left[\frac{1}{4}\frac{\partial^2}{\partial s^{2}}\sum\limits_{j=1}^{n}\left({x}_j^2+{y}_j^2\right)+\frac{\partial}{\partial s}\sum\limits_{j=1}^{n}\left({x}_j\frac{\partial}{\partial y_j}-{y}_j\frac{\partial}{\partial x_j}\right) \right] v({x},{y},{s},t)+o(\epsilon^\alpha)\\
&\qquad=L v(x,y,s,t)+o(\epsilon^\alpha)=L\tilde v(x,y,s,t)+o(\epsilon^\alpha).
\end{align*}
\end{proof}

\begin{proof}[Proof of Theorem \ref{3}]
In order to prove the theorem by a comparison argument we first look for a supersolution. Let $\overline w$ be given by
\begin{align}\label{4.2}
\overline w(z,s,t):=K_1\epsilon^\alpha t+K_2\epsilon.
\end{align}
For $(z,s,t) \in \Omega\times[0,T]$ we have $\tilde L_\epsilon(\overline w)(z,s,t)=0$, and if $K_1$ is large by Lemma \ref{lem4} and equation \eqref{pro2.2}:
\begin{align}\label{4.3}
\overline w_t(z,s,t)-\tilde L_\epsilon(\overline w)(z,s,t)=K_1\epsilon^\alpha\geq F_\epsilon(z,s,t)=w^\epsilon_t(z,s,t)-\tilde L_\epsilon(w^\epsilon)(z,s,t).
\end{align}
Since
$$
G(z,s, t) = o(\epsilon) \qquad \text{for }(z,s) \text{ such that dist}((z,s), \partial\Omega)\leq\epsilon d,
$$
choosing $K_2$ large, we obtain
\begin{align}\label{4.4}
\overline w(z,s,t)\geq w^\epsilon_t(z,s,t),
\end{align}
for $(z,s)\notin \Omega$ such that $\text{dist}((z,s), \partial\Omega)\leq\epsilon $ and $t\in [0, T ]$. Moreover, it is clear that
\begin{align}\label{4.5}
\overline w(z,s,0)=K_2\epsilon>0= w^\epsilon_t(z,s,0).
\end{align}
By \eqref{4.3}, \eqref{4.4} and \eqref{4.5} we can apply the comparison result, Corollary \ref{comp}, and conclude that 
\begin{align}\label{4.6}
w^\epsilon(z,s, t)\leq \overline{w}(z,s, t) =K_1\epsilon^\alpha t+K_2\epsilon.
\end{align}
In a similar way we prove that $\underline w(x, t) =- K_1\epsilon^\alpha t-K_2\epsilon$ is a subsolution and hence,
\begin{align}\label{4.7}
w^\epsilon(z,s, t)\geq \underline{w}(z,s, t)=-K_1\epsilon^\alpha t-K_2\epsilon.
\end{align}
Therefore by \eqref{4.6}, \eqref{4.7} and since $0<\alpha<1$, we get
\begin{align*}
\sup_{t\in[0,T]}\|v - u^ \epsilon\|_{L^\infty(\Omega)}=\sup_{t\in[0,T]}\|w^ \epsilon\|_{L^\infty(\Omega)}\leq C(T)\epsilon^\alpha.
\end{align*}
This proves the theorem.
\end{proof}

{\bf Acknowledgements.} Partially supported by CONICET and Secyt-UNC.

We want to thank L. V. Saal and U. Kaufmann for several interesting discussions.

\end{document}